\documentclass[12pt]{amsart}

\usepackage{amsfonts}
\usepackage{amssymb}
\usepackage{graphicx}
\usepackage{amsmath,mathtools}
\usepackage{latexsym}
\usepackage{xypic}
\usepackage{mathrsfs}%scrfont}
\usepackage{enumitem} %http://ctan.org/pkg/enumitem
\usepackage{braket}
\usepackage[margin=1.5in]{geometry}
\usepackage{esint}
\usepackage{dsfont}
\usepackage{amsthm}
\usepackage{tikz}

\usepackage{tikz-cd}
\usepackage{hyperref}

\setlist{itemsep=3pt}

\allowdisplaybreaks

\newtheorem{prop}{Proposition}
\newtheorem{theo}[prop]{Theorem}
\newtheorem{lemm}[prop]{Lemma}

\newtheorem*{claim}{Claim}

\theoremstyle{definition}
\newtheorem{defi}[prop]{Definition}

\newtheorem{quest}[prop]{Question}
\newtheorem{rema}[prop]{Remark}

\numberwithin{prop}{section}
 
\theoremstyle{plain}
\newtheorem{taggedasmp}{Assumption}
\makeatletter
\newcommand{\setasmptag}[1]{% \settheoremtag{<tag>}
	\renewcommand{\thetaggedasmp}{#1}% Redefine it to a fixed value
	}
	\makeatother

\newcommand{\CC}{\mathbb{C}}

\newcommand{\HH}{\mathbb{H}}

\newcommand{\RR}{\mathbb{R}}

\newcommand{\ZZ}{\mathbb{Z}}

\renewcommand{\cH}{\mathcal H}

\newcommand{\cS}{\mathcal S}

\DeclareMathOperator{\supp}{supp}

\DeclareMathOperator{\Ric}{Ric}
\DeclareMathOperator{\BiRic}{BiRic}

\DeclareMathOperator{\inte}{int}

\DeclareMathOperator{\Tor}{Tor}

\DeclareMathOperator{\diam}{diam}

\DeclareMathOperator{\dist}{dist}

\usepackage{graphicx}

\newcommand{\eps}{\varepsilon}

\title[Topology of stable minimal hypersurfaces]{On the topology of stable minimal hypersurfaces in a homeomorphic $S^4$}

\begin{document}

\author{Chao Li}
\address{Courant Institute, New York University, 251 Mercer St, New York, NY 10012, USA}
\email{chaoli@nyu.edu}

\author{Boyu Zhang}
\address{Department of Mathematics, The University of Maryland at College Park, Maryland, 20742, USA}
\email{bzh@umd.edu}

\maketitle

\begin{abstract}
	We construct stable minimal hypersurfaces with simple topology in certain compact $4$-manifolds $X$ with boundary, where $X$ embeds into a smooth manifold homeomorphic to $S^4$. For example, if $X$ is equipped with a Riemannian metric $g$ with positive scalar curvature, we prove the existence of a stable minimal hypersurface $M$ that is diffeomorphic to either $S^3$ or a connected sum of $S^2\times S^1$'s, ruling out spherical space forms in its prime decomposition. These results imply new theorems on the topology of black holes in four dimensions. The proof involves techniques from geometric measure theory and $4$-manifold topology.
\end{abstract}

\section{Introduction}\label{section.introduction}

Given a compact Riemannian manifold $X$ with nontrivial topology,  a fundamental question is to construct closed minimal surfaces \textit{with controlled topology} in $X$. The classical result by Sacks-Uhlenbeck \cite{SacksUhlenbeck} produces branched minimally immersed $S^2$. When $X$ is three dimensional, deep results by Meeks-Yau \cite{MeeksYau} and Meeks-Simon-Yau \cite{MeeksSimonYau} enable us to minimize area in homotopy and isotopy classes, respectively. In particular, one obtains area minimizing surfaces with controlled topology. These results have profound applications in geometry and topology.

It is generally impossible to control the topology of a $k$-dimensional minimal submanifold in $(X^{n+1},g)$ when $k>2$. Indeed, when $k>2$, White \cite{WhiteMapping,WhiteHomotopy} proved that the least area mapping in a given homotopy class generally does not give a smooth immersion. The Federer-Fleming compactness theorem constructs area-minimizing currents in an integral homology class. These geometric objects enjoy much better regularity properties: for example, when $k=n\le 6$, an area-minimizing hypersurface $M^{n}$ in $(X^{n+1},g)$ is smooth. On the other hand, just knowing that $M$ is in a given homology class does not place much restrictions on its topological type. 

The primary scope of this paper is to construct stable (or locally minimizing) minimal hypersurfaces with controlled topology in certain $4$-manifolds under natural curvature conditions. Let $(X^4,g)$ be a smooth orientable $4$-manifold. To motivate the discussion, recall that positive Ricci curvature of $g$ implies the nonexistence of two-sided stable minimal hypersurfaces. On the other hand, if $g$ has positive scalar curvature (abbreviated as PSC in the sequel), then the Schoen-Yau descent argument implies that a two-sided stable minimal $M^3$ is Yamabe positive. Therefore, by Schoen-Yau \cite{SY79}, Gromov-Lawson \cite{GromovLawson80} and Perelman, $M^3$ is diffeomorphic to a connected sum of spherical space forms and $S^2\times S^1$'s. Conversely, any such $M^3$ may arise as a stable (or locally area-minimizing) hypersurface in certain PSC $(X^4,g)$.

In our first result, we prove that one obtains significantly better topology control of stable minimal hypersurfaces in a PSC $4$-manifold $(X^4,g)$, provided that $X^4$ is itself topologically simple. We introduce our important topological assumption on $X$:

\setasmptag{($\ast$)}
\begin{taggedasmp}
	\label{topo.assumption.on.domain}
	$X$ is a connected compact manifold with boundary, and there exists a smooth embedding $\iota:X \to S^4$ such that $S^4\setminus \iota (X)$  has at least two connected components and at least one of them is simply connected.
\end{taggedasmp}

\begin{rema}
In fact, if we replace $S^4$ by any smooth manifold that is homeomorphic to $S^4$, the results in this paper still hold without any change. Since there is no known example of a smooth manifold that is homeomorphic but not diffeomorphic to $S^4$, we only state Assumption \ref{topo.assumption.on.domain} with $S^4$ to simplify notation. If needed, one may replace $S^4$ in Assumption \ref{topo.assumption.on.domain} with any smooth manifold that is homeomorphic to $S^4$. 
\end{rema}

%\begin{multline}
%	\label{topo.assumption.on.domain}
%	\tag{*} \text{There exists a continuous embedding } \iota:X \to S^4, \text{ such that the closure } 
%	\\ \text{ of  } S^4\setminus \iota (X) \text{ is a topological manifold} 
%	\text{ with boundary,} \text{ has at least two connected} 
%	\\ \text{ components, and at least one of the connected component is simply connected.}
%\end{multline}

Since $S^4$ is simply connected, each connected component of $S^4\setminus \iota(X)$ is bounded by exactly one component of $\partial (\iota(X))$. Thus, any such $X$ has at least two boundary components. Also, $X$ is necessarily orientable. Examples of such $X$ include the complement on $S^4$ of finitely many  open domains $D_1,\cdots,D_n$ whose closures are disjoint and at least one of $\{D_j\}$ is simply connected. In particular, $S^3\times [0,1]$ and $\left(\#^l S^2\times S^1\right)\times [0,1]$ satisfy \ref{topo.assumption.on.domain}.

\begin{theo}\label{theo.main.psc}
	Suppose $(X^4,g)$ is a smooth Riemannian manifold satisfying \ref{topo.assumption.on.domain}, such that $g$ has positive scalar curvature and $\partial X$ is $g$-weakly mean convex. Then $(X^4,g)$ contains a two-sided embedded stable minimal hypersurface that is diffeomorphic to $S^3$ or a connected sum of $S^2\times S^1$'s.
\end{theo}

In other words, Theorem \ref{theo.main.psc} rules out all nontrivial spherical components in the prime decomposition of $M$. This was unknown even in the simplest case when $X = S^3\times [0,1]$. We will see in Section \ref{section.example} that there exists $(X^4,g)$ satisfying the assumptions of Theorem \ref{theo.main.psc} where the only stable minimal hypersurfaces are diffeomorphic to $S^2\times S^1$. We note here that the assumption that $\partial X$ has at least two connected components is only used to guarantee that each any such component $Y$ is homologically nontrivial, and this assumption may be removed if we know the existence of a stable minimal hypersurface in the class $[Y]$ for some other reasons.

In four dimensions, there is an interesting notion of curvature  -- the bi-Ricci curvature -- that interpolates between the Ricci and the scalar curvature. 

\begin{defi}\label{defi.biricci}
	Given a Riemannian manifold $(X^n,g)$, $p\in X$ and orthonormal $e_1, e_2\in T_p X$, the bi-Ricci curvature $p$ of $\{e_1,e_2\}$ is defined as
	\[\BiRic(e_1,e_2): = \Ric(e_1) + \Ric(e_2) - \sec(e_1,e_2).\]
\end{defi}

Equivalently, $\BiRic_g (e_1,e_2) = R_g - \sec_g(e_3,e_4)$ for an orthonormal basis $\{e_1,\cdots,e_4\}$. Here $R_g$ denotes the scalar curvature of $g$. Observe that positive bi-Ricci curvature implies PSC. Manifolds with positive bi-Ricci curvature were first considered by Shen-Ye \cite{ShenYe}. Bi-Ricci curvature and its generalizations have been studied systematically by Brendle-Hirsch-Johne \cite{BrendleHirschJohne} recently, and have deep applications in the stable Bernstein problem for minimal hypersurfaces \cite{CLMS24bernstein,Mazet24bernstein}. In \cite{ShenYe}, it is proved that a two-sided stable minimal hypersurface $M^3$ in $(X^4,g)$ with positive bi-Ricci curvature has positive Ricci curvature in the spectral sense, and hence (by the resolved Poincar\'e conjecture) is diffeomorphic to a spherical space form. In our second result, we prove the existence of such $M$ that is actually diffeomorphic to $S^3$, provided that $X^4$ satisfies \ref{topo.assumption.on.domain}.

\begin{theo}\label{theo.main.biricci}
	Suppose $(X^4,g)$ satisfies all assumptions in Theorem \ref{theo.main.psc}, and additionally $\BiRic_g>0$. Then $(X^4,g)$ contains a two-sided embedded stable minimal hypersurface that is diffeomorphic to $S^3$.
\end{theo}

We expect such  existence results to be applicable for further investigations on $4$-manifolds.

\subsection{Strategy of proof}

In \cite{White1992topology}, White illustrated how to obtain some mild controls on the topology of solutions to the Plateau problem in a ball. His basic observation is that stable minimal hypersurfaces at extremal positions tend to bound  topologically simple regions. We adopt this general principle in our construction. Given $(X^4,g)$ satisfying \ref{topo.assumption.on.domain}, let $D$ be a simply connected component of $S^4\setminus \iota(X)$, and $Y$ be the boundary component of $X$ such that $\iota(Y) = \partial D$. The key is to consider, among stable minimal hypersurfaces homologous to $Y$, the one $M$ that is the \textit{nearest} to $Y$. Through a novel minimization argument in general covering spaces (that are possibly neither compact or normal), we establish some conditions on the fundamental group of the region between $Y$ and $M$ (see Proposition \ref{prop.topology.extremal.minimal}). This part of the proof is very general (e.g. does not depend on the dimension or the topological assumptions of $X$), and we expect it to be useful for other problems.

The second part of the argument relies heavily on techniques in $4$-manifold topology. Since $X$ satisfies \ref{topo.assumption.on.domain}, Proposition \ref{prop.topology.extremal.minimal} implies that $M$ admits a locally flat embedding into $S^4$ that bounds a simply connected region. The embedding problem of three-manifolds into $S^4$ has been extensively investigated, and a collection of classical results may be found in a recent survey by Hillman \cite{hillman2024locally}. These arguments are extremely well suited to study the case when $M = (\#S^3/\Gamma_i)\#(\#^l (S^2\times S^1))$, which holds in this setting because of the Yamabe positivity property. 
%by Schoen-Yau \cite{SY79}, Gromov-Lawson \cite{GromovLawson80} and Perelman. 
It follows from classical arguments that the only possible nontrivial spherical prime factor in $M$ is the Poincar\'e homology sphere. We then use Floer-theoretic invariants and a deep result of Taubes \cite{taubes1987gauge} to rule it out.

\subsection{A theorem on the topology of black holes}

Consider an asymptotically flat manifold $(X^n,g)$ with nonnegative scalar curvature, and let $E$ be an end of $X$. The \textit{outermost apparent horizon} $M$ of $E$, if non-empty, is defined as the outermost minimal hypersurface in $E$. Hawking's classical black hole topology theorem \cite{HawkingEllis} states that when $n=3$, $M$ is a disjoint union of $S^2$. Generalizations in higher dimensions due to Cai, Galloway and Schoen \cite{CaiGalloway2001,GallowaySchoen2006,Galloway2008} conclude that if $M$ is smooth, then it is a stable minimal hypersurface, and is Yamabe positive. 

Our proof of Theorem \ref{theo.main.psc} reveals the topology of stable minimal hypersurfaces which is outermost with respect to a mean convex barrier. Thus it directly applies to studying the topology of apparent horizons in asymptotically flat four-manifolds, provided that $X^4$ satisfies \ref{topo.assumption.on.domain}.

\begin{theo}\label{theo.gr}
	Suppose $(X^4,g)$ is a smooth asymptotically flat manifold with nonnegative scalar curvature, and $X$ is diffeomorphic to the interior of a manifold that satisfies \ref{topo.assumption.on.domain}. Then the outermost apparent horizon of each asymptotically flat end of $X$ is diffeomorphic to a disjoint union of $S^3$ or a connected sum of $S^2\times S^1$'s.
\end{theo}

In fact, our proof only requires that the trapped region (the unbounded component separated by an apparent horizon on an asymptotically flat end) is a subset of $\RR^4$, for which the topological assumption \ref{topo.assumption.on.domain} is a natural sufficient condition. In particular, if $X^4$ is diffeomorphic to the complement of finitely many compact sets on $S^4$, its apparent horizon of each asymptotically flat end is $S^3$ or connected sums of $S^2\times S^1$. We expect this conclusion to be sharp: Dahl-Larsson \cite{DahlLarsson2019} constructed examples of apparent horizons on a subset of some asymptotically flat $\RR^4$ that are diffeomorphic to the disjoint union of $l$ copies of $S^2\times S^1$ for each $l\ge 1$. We note here that Theorem \ref{theo.gr} also applies to other ALF or ALG four dimensional gravitational instantons with nonnegative scalar curvature, provided that $X$ is diffeomorphic to the interior of a manifold that satisfies \ref{topo.assumption.on.domain}. These include certain manifolds admitting a fibered end with fibers diffeomorphic to $S^2\times S^1$. On the other hand, the topology of apparent horizons may be more complicated (e.g. lens spaces $L(p,q)$) if $X$ fails to satisfy \ref{topo.assumption.on.domain}, as illustrated by the beautiful black lenses examples constructed by Khuri-Rainone \cite{KhuriRainone2023}. Let us remark that all these examples are simply connected $4$-manifolds, which (after suitable compactifications) are homeomorphic to connected sums of $S^2\times S^2$, $\CC P^2$ and $-\CC P^2$. It would be interesting to investigate whether other spherical space forms (e.g. the Poincar\'e homology sphere) may appear as the outermost apparent horizon for a four dimensional gravitational instanton with nonnegative scalar curvature.

\subsection{Organization of the paper}
The paper is organized as follows. In Section \ref{section.minimal.hypersurface}, we establish general topological constraints for minimal hypersurfaces at extremal positions. The key result is Proposition \ref{prop.topology.extremal.minimal}, which relies on a very general existence result for homologically area-minimizing hypersurfaces in any covering space of a compact manifold with boundary. In Section \ref{section.space.forms}, we focus on four-dimensional manifolds and use topological techniques to rule out nontrivial spherical space forms in the prime decomposition of the stable minimal hypersurface, finishing the proof of Theorem \ref{theo.main.psc} and Theorem \ref{theo.main.biricci}. Section \ref{section.black.hole} is devoted to the black hole topology theorem. Finally, in Section \ref{section.example}, we discuss an extension of our results for solutions to the Plateau problem in a homeomorphic $D^4$, examples of PSC embeddings  of connected sums of lens spaces into $S^4$, and the Dahl-Larsson example of $S^2\times S^1$ apparent horizons. Some natural questions are also discussed and posed.

\subsection*{Acknowledgement} The authors are grateful to Claude LeBrun, Christos Mantoulidis, Rick Schoen and Brian White for stimulating conversations on various topics in this paper. We thank Daniel Ruberman for explaining the Zeeman construction to us, and Marcus Khuri for discussions on black hole topology and for patiently answering our questions on \cite{KhuriRainone2023}. C.L. is supported by an NSF grant (DMS-2202343), a Simons Junior Faculty Fellowship and a Sloan Fellowship. B.Z. is supported by an NSF grant (DMS-2405271) and a travel grant from the Simons Foundation.

\section{Stable minimal hypersurfaces of extremal positions}\label{section.minimal.hypersurface}

In this section, we obtain some preliminary topological constraints on stable minimal hypersurfaces in a compact Riemannian manifold $(X^{n+1},g)$ with weakly mean convex boundary. Assume $n\le 6$ for the regularity of area-minimizing hypersurfaces \footnote{The conclusions of this section should hold in higher dimensions as well, thanks to Simon's maximum principle \cite{Simon1987maximum}. We do not pursue this direction.}. Let $Y$ be a connected component of $\partial X$. Suppose that
\[I:=\inf\{\cH^n(M): M^n \text{ is a hypersurface in }X \text{ homologous to }Y \}>0.\]
This holds when $[Y]\in H^n(X, \ZZ)$ is nonzero, for example when $\partial X$ has at least two connected components. Since $\partial X$ is weakly mean convex, $I$ can be realized by a (possibly disconnected) smooth stable minimal hypersurface $M^n\subset X^{n+1}$. Consider 
\begin{multline*}
	\cS = \{M^n\subset X^{n+1}: M \text{ is an embedded stable minimal hypersurface} \\
	\text{ homologous to }Y, \cH^n(M)\le  \cH^n(Y) \}.
\end{multline*}

By standard curvature estimates for stable minimal hypersurfaces \cite{SSY,SchoenSimon}, $\cS$ is compact in the $C^\infty$ topology. For each $M\in \cS$, let $\Omega_M$ be the $(n+1)$-dimensional manifold such that $\partial \Omega_M = Y-M$ (regarded as currents).
%cut $X$ open along $M$ to obtain a new compact Riemannian manifold with boundary, and let $\Omega_M$ be the connected component of the resulting manifold such that $Y\subset \Omega_M$.} 
The key result of this section is the Proposition \ref{prop.topology.extremal.minimal}.  This is a nontrivial extension of \cite[Theorem 1]{White1992topology}.

\begin{prop}\label{prop.topology.extremal.minimal}
	Given a compact oriented Riemannian manifold $(X^{n+1},g)$ as above. There exists $M\in \cS$ such that either $M=Y$, or $\Omega_M$ is connected and satisfies
	\[\cH^{n+1}(\Omega_M) = \min\{\cH^{n+1}(\Omega_{M'}): M'\in \cS \}.\]
	Moreover, for every $y\in Y$, the inclusion $i: Y\to \Omega_M$  induces a surjective mapping
	\[i_*: \pi_1(Y,y) \to \pi_1(\Omega_M,y).\]
\end{prop}
Note that since $Y$ is connected, the surjectivity of $i_*$ does not depend on the choice of the base point $y$. 

The rest of this section is devoted to the proof of Proposition \ref{prop.topology.extremal.minimal}. 
Since $\cS$ is compact in $C^\infty$ convergence, there exists $M\in \cS$ such that $\cH^{n+1}(\Omega_M) = \min\{\cH^{n+1}(\Omega_{M'}):M'\in \cS \}$. Assume that $M\ne Y$. Then by the strong maximum principle, $M$ is contained in the interior of $X$. Also, observe we necessarily have that $\Omega_M$ is connected: otherwise discarding all connected components of $\Omega_M$ except for the one containing $Y$ yields an $\Omega_{M'}$ with smaller volume. 

Suppose, for the sake of contradiction, that $i_*$ is not surjective. Take $y\in Y$, take the connected covering space $\pi: (\tilde \Omega_M, \tilde y) \to (\Omega_M, y)$ such that
\[\pi_*(\pi_1(\tilde \Omega_M, \tilde y)) = i_*(\pi_1(Y,y)).\]
Denote by $\tilde g = \pi^* g$ the covering metric and $Y_0$ the connected component of $\pi^{-1}(Y)$ that contains $\tilde y$. Since $i_*$ is not surjective, $\pi^{-1}(Y)$ has more than one connected components. Our basic idea is to prove that $[Y_0]\in H_n(\tilde \Omega_M)$ is nontrivial, and that we may minimize area in $[Y_0]$ to find another area-minimizing hypersurface in $\tilde \Omega_M$, so that its projection in $\Omega_M$ gives a stable minimal hypersurface in $\Omega_M$ that bounds a region with strictly smaller volume, contradicting the choice of $M$.

However, a key difficulty is that generally, a minimizing sequence in a complete noncompact manifold does not necessarily converge. Note that we cannot assume that the covering space $\tilde{\Omega}_M$ is normal, and we cannot expect our area-minimizing hypersurface to be connected. To proceed, we establish some topological properties of the covering space of a compact manifold with boundary, and we carefully carry out the construction of an area-minimizing hypersurface. This argument seems novel and we expect it to be useful for other applications.

We start by introducing some notations. 
If $X$ is a manifold with boundary, let $\inte(X)$ denote the interior of $X$.

\begin{defi}\label{defi.intersection.number}
	Suppose $X$ is a smooth oriented $(n+1)$--manifold with boundary, $a \in H_n(X)$, and $(\gamma,\partial \gamma)\subset (X, \partial X)$ is a smooth oriented properly embedded compact 1-manifold. 
	Let $\omega_\gamma$ be a Thom form of $\gamma$. Define the intersection number of $a$ and $\gamma$ to be 
	$\langle a, [\omega_\gamma] \rangle,$
	where $[\omega_\gamma] \in H^n_{dR}(\inte(X)) \cong H^n(X,\mathbb{R})$ denotes the cohomology class of $\omega_\gamma$. 
\end{defi}
%\begin{defi}\label{defi.intersection.number}
%	Suppose $a \in H_n(\tilde{\Omega}_M)$, and $(\gamma,\partial \gamma)\subset (\tilde{\Omega}_M, \partial \tilde{\Omega}_M)$ is a smooth properly embedded compact 1-manifold intersecting $\partial\tilde\Omega_M$ transversely. 
%	Let $\omega_\gamma$ be the Thom form of $\gamma$. Define the intersection number of $a$ and $\gamma$ to be 
%	$\langle a, [\omega_\gamma] \rangle,$
%	where $[\omega_\gamma] \in H^n_{dR}(\inte(\tilde{\Omega}_M)) \cong H^n(\tilde{\Omega}_M,\mathbb{R})$ denotes the cohomology class of $\omega_\gamma$. 
%\end{defi}

\begin{rema}
	If $a = [M]$ for a smooth embedded compact $n$-manifold $M$, then the intersection number of $a$ and $\gamma$ equals $\int_M \omega_\gamma$. If we further assume that $M$ intersects $\gamma$ transversely, then the integral is equal to the signed counting of the number of intersection points between $M$ and $a$. 
\end{rema}

\begin{rema}
Alternatively, the intersection number of $a$ and $\gamma$ can be defined without using differential forms as follows. Let $[\gamma]\in H_1(X,\partial X)$ be the fundamental class of $\gamma$. 
The inclusion map induces an isomorphism between $H_n(\inte(X))$ and $H_n(X)$, so we may view $a$ as an element of $H_n(\inte(X))$. Poincar\'e duality gives an isomorphism 
\[
PD: H_n(\inte(X)) \to H_c^1 (\inte(X)),
\]
where $H_c^*$ denotes the compactly supported cohomology. By the definition of $H_c^1 (\inte(X))$, there exists a compact set $C\subset \inte(X)$ such that $PD(a)$ is represented by an element 
\[
a' \in H^1 (\inte(X), \inte(X)\setminus C).
\]
The intersection number of $a$ and $\gamma$ is equal to the pairing of $a'$ with the image of $[\gamma]$ under the map 
\[
H_1(X, \partial X) \to H_1 (X, X\setminus C) \xrightarrow{\cong} H_1 (\inte(X), \inte(X)\setminus C),
\]
where the first arrow is induced by inclusion, and the second map is the excision isomorphism. 
\end{rema}

The key topological property that enables us to construct homologically minimizing hypersurfaces in $\tilde \Omega_M$ is the following Lemma.

\begin{lemm}
	\label{lem_intersection_nonzero}
	Suppose $0\neq a \in H_n(\tilde{\Omega}_M)$. Then there exists a smooth oriented properly embedded compact 1-manifold $(\gamma,\partial \gamma)\subset (\tilde{\Omega}_M, \partial \tilde{\Omega}_M)$ such that the intersection number of $\gamma$ and $a$ is non-zero. 
\end{lemm}

\begin{rema}
	We emphasize that Lemma \ref{lem_intersection_nonzero} holds for all covering spaces of a compact manifold with nonempty boundary, and it relies on the existence of boundary in an essential way. For example, if $M$ is a closed $n$--manifold, then $[M]\in H_n(\mathbb{R}\times M)$ has intersection number zero with every compact $1$--submanifold of $\mathbb{R}\times M$. 
\end{rema}

\begin{proof}[Proof of Lemma \ref{lem_intersection_nonzero}]
	If $\tilde{\Omega}_M$ is compact, then the desired result follows from Poincar\'e duality. In the following, we assume that $\tilde{\Omega}_M$ is non-compact. Recall that $\tilde g$ denotes the pull-back metric on $\tilde \Omega_M$ via the covering map.
	
	Let $x$ be a fixed point in $\inte(\tilde{\Omega}_M)$, and let $B_\rho(x)$ be the geodesic ball with radius $\rho$ centered at $x$. 
	For generic $\rho$, we have $\partial B_\rho(x)$ is a properly embedded, compact, smooth submanifold of $\tilde{\Omega}_M$ with codimension $1$. 
	Since every homology class is represented by finitely many singular simplices, for $\rho$ sufficiently large, the homology class $a$ is contained in the image of $H_n(B_\rho(x)) \to H_n(\tilde{\Omega}_M)$. 
	From now, let $\rho$ be a fixed number that is both sufficiently large and generic so that the above properties hold.

	Since $\partial B_\rho(x)$ has only finitely many connected components, the set $\tilde{\Omega}_M\setminus B_\rho(x)$ has at most finitely many connected components. Let $C_1,\dots,C_n$ be the closures of the connected components of $\tilde{\Omega}_M\setminus B_\rho(x)$. Then each $C_i$ is a manifold with corners, where the codimension of each corner stratum is at most $2$. 
	
	\begin{claim}
		For each non-compact $C_i$, we have $\partial C_i\cap \partial \tilde{\Omega}_M \neq \emptyset$.
	\end{claim} 
	\begin{proof}[Proof of the claim]
		Assume there exists a $C_i$ that is non-compact and $\partial C_i\cap \partial \tilde{\Omega}_M = \emptyset$. Then $C_i$ is a manifold with boundary, and $\partial C_i\subset \partial B_\rho(x)$. 
		
		Since $\Omega_M$ is a compact connected Riemannian manifold with a non-empty boundary, it has the following two properties. These are straightforward extensions of the Hopf--Rinow theorem to manifolds with boundary:
		\begin{enumerate}
			\item Every geodesic $[0,\epsilon)\to \inte(\Omega_M)$  either extends to a geodesic $[0,+\infty)\to \inte(\Omega_M)$ or to a geodesic $[0,l]\to \Omega_M$ such that $[0,l)$ is mapped to $\inte(\Omega_M)$ and $l$ is mapped to $\partial \Omega_M$. 
			\item For every $p\in \inte(\Omega_M)$, there exists $q\in \partial \Omega_M$ and a geodesic $\gamma$ from $p$ to $q$, such that $\gamma\perp \partial \Omega_M$ at $q$, and the length of $\gamma$ equals the distance between $p$ and $\partial \Omega_M$. 
		\end{enumerate}
		
		By Property (1) above, a geodesic on $C_i$ either extends indefinitely or intersects $\partial C_i$ in finite time. Fix a constant $R$ such that 
		$R > \diam \Omega_M.$
		Since $C_i$ is non-compact, there exists $p\in C_i$ such that the distance from $p$ to $\partial C_i$ is at least $R$. Therefore, every geodesic starting at $p$ can be extended to a geodesic in $C_i$ with length at least $R$. On the other hand, let $\pi(p)$ be the image of $p$ on $\Omega_M$. By Property (2) above, there exists a geodesic on $\Omega_M$ from $\pi(p)$ to $\partial \Omega_M$ with length at most $\diam \Omega_M$. It lifts to a geodesic on $\tilde{\Omega}_M$, which starts at $p$ and ends on $\partial \tilde{\Omega}_M$ and has length at most $\diam \Omega_M<R$.  This yields a contradiction. 
	\end{proof}
	
	Now we finish the proof of the lemma using the claim. Let $\hat \Omega\subset \tilde{\Omega}_M$ be the union of $B_\rho(X)$ and all the $C_i$'s which are compact. Then $\hat \Omega$ is a compact manifold with corners, and $a$ is in the image of $H_3(\hat \Omega) \to H_3(\tilde{\Omega}_M)$. 
	
	Let $\hat a\in H_3(\hat \Omega)$ be a preimage of $a$. Since $a\neq 0$, we know that $\hat a\neq 0$. Note that $\hat \Omega$ is a compact smooth manifold with boundary after smoothing the corners. By Poinca\'e duality, there exists a properly embedded smooth 1-manifold $(\hat \gamma,\partial \hat \gamma)\subset (\hat \Omega, \partial \hat \Omega)$ such that the intersection number of $\hat a$ with $\hat \gamma$ is non-zero, and we may perturb $\hat \gamma$ such that $\partial \gamma$ does not intersect the corners of $\partial\hat \Omega$. 
	
	For every $q\in \partial \hat \gamma$ with $q\notin \partial\tilde{\Omega}_M$, there exists a unique $C_i$ such that $C_i$ is non-compact and $q\in \partial C_i$. By the above claim, there exists $q'\in \partial C_i\cap \partial \tilde{\Omega}_{M}$. Let $\gamma_q$ be a properly embedded arc in $C_i$ that connects $q$ and $q'$. After perturbing $\gamma$ and $\gamma_q$ near $q$, we may further assume that $\gamma\cup \gamma_q$ is smooth near $q$. 
	
	Let 
	\[
	\gamma = \hat \gamma\cup(\cup_{q} \gamma_q),
	\]
	where the union takes over all $q$ such that $q \in \partial \hat \gamma, q\notin \partial \tilde{\Omega}_M$. If $n\ge 2$, then after a generic perturbation, $\gamma$ is a properly embedded $1$--manifold in $\tilde{\Omega}_M$. If $n=1$, then after a generic perturbation $\gamma$ is a properly immersed $1$--manifold with transverse self-intersections, and we can resolve the self-intersection of $\gamma$ to obtain an properly embedded $1$--manifold in $\tilde{\Omega}_M$ with the same homology class.
	
	Since all $\gamma_q$ are disjoint from $\inte(\hat \Omega)$, a Thom form of $\gamma$ on  $\tilde\Omega_M$ restricts to a Thom form of $\hat \gamma$ on $\hat \Omega$. So the intersection number of $a$ and $\gamma$ in $\tilde\Omega_M$ is equal to the intersection number of $\hat a$ and $\hat \gamma$ in $\hat \Omega$, which is non-zero by the definition of $\hat \gamma$. Hence the lemma is proved. 
\end{proof}

We are now ready to continue the proof of Proposition \ref{prop.topology.extremal.minimal}.

\begin{proof}[Proof of Proposition \ref{prop.topology.extremal.minimal}, continued]
	Consider the minimization problem
	\[\tilde I = \inf\{\cH^n(N): N  \in [Y_0] \}. \]
	Let $(\gamma_0,\partial\gamma_0)\subset (\tilde \Omega_M, \partial \tilde \Omega_M)$ be the compact embedded curve constructed in Lemma \ref{lem_intersection_nonzero}.
	
	Let $\omega_0$ be the Thom form of $\gamma_0$ as in Definition \ref{defi.intersection.number}. 
	By definition, $\omega_0$ is compactly supported in a neighborhood of $\gamma_0$.
	Then any $n$-cycle $N\in [Y_0]$ satisfies that $\int_N \omega_0 \neq 0$.
	
	For $\rho \gg 1$ consider the $\tilde g$-geodesic ball $B_\rho (\tilde y)$, and let $S_\rho(\tilde y)$ be the $\tilde g$-geodesic sphere. Choose $\rho$ large enough such that $[Y_0]\ne 0\in H_n(B_\rho(\tilde y))$. Perturbing $\rho$ a bit if necessary, we assume that $S_\rho(\tilde y)$ meets $\partial \tilde \Omega_M$ transversely. Deform the metric $\tilde g$ to $\tilde g'$ in a small neighborhood of $S_\rho (\tilde y)$, such that $S_\rho(\tilde y)$ is $\tilde g'$ strictly mean convex and meets $\partial \tilde \Omega_M$ orthogonally. We then may minimize the $\cH^n$ volume (with respect to $\tilde g'$) in the nontrivial homology $[Y_0]$ in $(B_\rho(\tilde y), \tilde g')$, and obtain a possibly disconnected area minimizing hypersurface $N_\rho$. Since
	\[\int_{N_\rho} \omega_0\ne 0, \]
	each $N_\rho$ has a non-trivial intersection with the compact set $\supp \omega_0$. Consider $N_\rho^0$ the union of all connected components of $N_\rho$ that intersect $\supp\omega_0$. Since $N_\rho^0$ has uniformly bounded $\cH^n$-volume, standard curvature estimates imply that they subsequentially (which we do not relabel) $C^\infty$ graphically converge to a limit $N^0$ (possibly with integer multiplicity at this moment). Note that $N^0$ is compact: since $\tilde \Omega_M$ has bounded geometry, by the monotonicity formula, there exist $r_0>0$ and $\eps_0>0$ depending only on $\Omega_M$ such that if $x\in N^0$, then the intersection of $N^0$ and $B_{r_0}(x)$ has $\cH^n$-volume at least $\eps_0$.
	
	%its intersection with any fundamental domain in the covering space $\tilde \Omega_M$ has a uniform $\cH^n$-volume lower bound by the monotonicity formula. 
	Therefore we conclude that $\{N_\rho^0\}_\rho$ is also compactly supported, and the convergence to $N^0$ in fact holds as currents. In particular, $N_\rho^0$ is homologous to $N^0$ for sufficiently large $\rho$.
	
	If $N^0$ is homologous to $Y_0$, we are done. Otherwise, consider the homology class $[Y_0] - [N_0]$. By Lemma \ref{lem_intersection_nonzero}, there exists a properly embedded compact curve $(\gamma_1,\partial \gamma_1)\subset (\tilde \Omega_M, \partial\tilde \Omega_M)$ that has a non-zero intersection number with $[Y_0] - [N_0]$. Denote by $\omega_1$ the Thom class of $\gamma_1$. Then for all sufficiently large $\rho$, the minimizing sequence $N_\rho - N_\rho^0$ satisfies that
	\[\int_{N_\rho - N_\rho^0}\omega_1 \ne 0.\]
	Let $N_\rho^1$ be the union of the connected components of $N_\rho - N_\rho^0$ that intersect the support of $\omega_1$. 
	By passing to a further subsequence (which again we do not relabel), the same argument as above finds a compact minimizing hypersurface $N_1$ as the limit of $\{N_\rho^1\}$.
	
	Inductively, assuming that we have constructed $N_0,\cdots,N_k$. If $N_0 + \cdots + N_k$ (taking the sum as currents) is not homologous to $Y_0$, we apply the above argument and construct a compact minimizing hypersurface $N_{k+1}$ from the minimizing sequence $\{N_\rho-(N_\rho^0+\cdots+N_\rho^k)\}$. Note that since $\tilde \Omega_M$ is the covering space of a compact manifold, any minimal hypersurface has a uniform lower bound on its $\cH^n$-volume. This implies that this construction terminates in finitely many steps, since $\{N_\rho \}$ has a uniform $\cH^n$-volume upper bound. We thus conclude that there exists a compact area minimizing hypersurface $N$ (possibly disconnected) in the homology class $[Y_0]$. 
	Note that $N$ is not contained entirely in $\partial \tilde \Omega_M$, as the only $n$-currents entirely contained in $\partial \tilde \Omega_M$ that are homologous to $[Y_0]$ are in the form $Y_0 + k \partial \tilde \Omega_M$ for some integer $k$ (if $\partial \tilde \Omega_M$ is non-compact, then we must have $k=0$). Since $Y_0$ is isometric to $Y$ and $\pi^{-1}(Y)$ has at least two connected components,  the mass of a current of the form $Y_0 + k \partial \tilde \Omega_M$ is at least $\cH^n(Y_0)$. We know $N\ne Y_0$ (as otherwise we would have picked $M=Y$), so $N$ cannot be contained entirely in $\partial \tilde{\Omega}_M$.

	Let $\hat N = \pi(N)$. If $\hat N$ is embedded, then it is homologous to $Y$ in $\Omega_M$, and bounds a region with smaller volume than $\cH^{n+1}(\Omega_M)$, contradicting the choice of $M$. If $\hat N$ is immersed, then it still represents the homology class $[Y]$ in $H_n(\Omega_M, \ZZ)$. We may then minimize the $n$-dimensional volume in an open neighborhood of $Y$ in $\Omega_M$, among hypersurfaces homologous to $Y$ that does not intersect $\hat N$ (note that immersed minimal hypersurfaces are weakly mean convex as self-intersections form an angle that is strictly less than $\pi$). This produces another stable minimal hypersurface $M'\in \cS$ such that $\cH^{n+1}(\Omega_{M'})<\cH^{n+1}(\Omega_M)$, contradiction.
\end{proof}

\begin{rema}\label{remark.simple.case.white}
	If the covering $(\tilde \Omega_M , \tilde y)\to (Y,y)$ is normal and the minimizer is connected, then we may also find the homologically minimizing hypersurface $N$ by translating (with deck transformations) a minimizing sequence to intersect a fixed compact set. This is the case treated by White \cite{White1992topology}.
\end{rema}

\begin{rema}
	Using Lemma \ref{lem_intersection_nonzero}, we actually proved the following existence result for stable minimal hypersurfaces. Let $(\Omega^{n+1},g)$ be a compact manifold with nonempty boundary, and $\partial \Omega$ is $g$-weakly mean convex. Let $(\tilde \Omega,\tilde g)$ be a covering space. Then any nonzero homology class $\alpha\in H_n(\tilde \Omega,\ZZ)$ can be represented by a compact area-minimizing hypersurface.
\end{rema}

\begin{rema}
	Although not used in this paper, we observe that an analogous statement also holds for minimal hypersurfaces with obstacles. That is, without assuming that $\partial X$ is weakly mean convex, we may still consider the minimization problem
	\[\inf \{\cH^n(\partial \Omega): \Omega \text{ is an open set of }X \text{ containing }Y  \}.\]
	Wang \cite[Corollary 3.3]{Wang2022obstacle} proved that local minimizers of this problem enjoys a $C^1$ compactness property. Thus, we may consider, among all open sets $\Omega$ containing $Y$ that locally minimizes $\cH^n(\partial \Omega)$, the one with the smallest volume. The proof of Proposition \ref{prop.topology.extremal.minimal} carries over to this situation verbatim and implies that $\pi_1(Y)\to \pi_1(\Omega)$ is surjective.
\end{rema}

\begin{rema}
	On the other hand, the proof of Proposition \ref{prop.topology.extremal.minimal} does not seem to easily extend to the case of stable constant mean curvature hypersurfaces, or more generally prescribed mean curvature surfaces. One key issue is that, in the universal cover $\tilde \Omega_M$, other connected components of $\pi^{-1}(Y)$ have a reversed bound of mean curvature, when regarded as a barrier for the minimization problem in the homology class of $Y_0$.
\end{rema}

\section{Eliminating space forms}\label{section.space.forms}

Now we focus on the case when $n+1=4$ and prove Theorem \ref{theo.main.psc}. The proof of Theorem \ref{theo.main.biricci} is similar but simpler. Given $(X^4,g)$ satisfying the assumptions of Theorem \ref{theo.main.psc}, let $D$ be a simply connected component of $S^4\setminus \iota(X)$, and $Y$ be the boundary component of $X$ such that $\iota(Y)=\partial D$. We apply Proposition \ref{prop.topology.extremal.minimal} to find an embedded stable minimal hypersurface $M$ and the connected region $\Omega_M$, 
%bounded by $Y$, $M$, 
such that $i_*: \pi_1(Y)\to \pi_1(\Omega_M)$ is surjective. If $M=Y$, we define $\Omega_M=Y$.

%{\color{red} Since $X$ embeds into $S^4$, the intersection number of every connected component of $M$ with every circle in $X$ must be zero. So, when cutting $X$ open along $M$, the two sides of each component of $M$ must belong to different connected components of the resulting manifold. As a result, we may view $\Omega_M$ as a subset of $X$.}

Denote by $A=\iota(\Omega_M)\cup_{Y} D$. Since $\pi_1(D) = 1$ and $i_*$ is surjective, it follows from the van Kampen theorem that $\pi_1 (A)=1$.
By the definitions of $\Omega_M$ and $A$, we know that $\iota(M)\subset \partial A$.

On the other hand, since $M$ is a stable minimal hypersurface in $(X^4,g)$ with $R_g>0$, $M$ itself admits a PSC metric. Therefore, each connected component of $M$ is diffeomorphic to a connected sum of spherical space forms and $S^2\times S^1$'s. The next basic lemma further implies that, in fact, each connected component of $\partial A$ individually bounds a simply connected domain of $S^4$. 

\begin{lemm}\label{lemm.reduction.to.connected}
	Suppose $A\subset S^4$ is a simply connected domain with smooth boundary, and let $M$ be a connected component of $\partial A$. Then $S^4\setminus M$ has two connected components. Let $A', B$ be the closures of the components of $S^4\setminus M$ such that $A\subset A'$. Then $A'$ is also simply connected.
\end{lemm}
\begin{proof}
	Let $D_1,\dots,D_k$ be the closures of the connected components of $S^4\setminus A$. Then 
	$\partial A$ is the disjoint union of $\partial D_1,\dots,\partial D_k$. We first show that every $\partial D_i$ is connected. Assume $\partial D_i$ is not connected, let $M_1,M_2$ be two connected components of $\partial D_i$, let $p$ be a point in the interior of $D_i$, let $q$ be a point in the interior of $A$. Then there exist arcs $\gamma_1,\gamma_2$ from $p,q$ such that $\gamma_i$ intersects $\partial A$ transversely at one point in each $M_i$ ($i=1,2$). As a consequence, the arcs $\gamma_1$ and $\gamma_2$ combine to define a loop in $S^4$ that intersects each of $M_1$ and $M_2$ transversely at one point. This contradicts the fact that $H_1(S^4)=0$.
	
	Let $\hat D_i = D_i/\partial D_i$. By definition, $\hat D_i$ is the quotient space of $D_i$ by collapsing $\partial D_i$ to a point. The space $\hat D_i$ may not be a manifold. 
	Let $S\subset \{1,\dots,k\}$ be non-empty, let $A_S = A \cup (\cup_{i\in S} D_i)$. 
	We claim that 
	\begin{equation}
	\label{eqn_AS_pi1_Di}
	\pi_1(A_S)\cong *_{i\in S}\, (\pi_1(\hat D_i)),
	\end{equation}
	where the right-hand side is the free product. 
	
	We prove \eqref{eqn_AS_pi1_Di} by induction on the number of elements of $S$. If $S=\emptyset$, the statement is trivial. Suppose \eqref{eqn_AS_pi1_Di} holds for all sets $S$ with $l$ elements, and consider a set $S'$ of the form $S' = S\cup \{i\}$ where $S$ has $l$ elements and $i\notin S$. Then $A_{S'} = A_S\cup D_i$. By the van Kampen theorem, 
	\[
	\pi_1(A_{S'}) \cong \pi_1(A_S) *_{\pi_1(\partial D_i)} \pi_1(D_i).
	\]
	Since the inclusion-induced map $\pi_1(\partial D_i) \to \pi_1(A_S)$ factors through $\pi_1(A)$, it is the trivial map. Therefore, we have
	\[
	\pi_1(A_{S'}) \cong \pi_1(A_S) *_{\pi_1(\partial D_i)} \pi_1(D_i) \cong \pi_1(A_S) * (\pi_1(D_i)/\pi_1(\partial D_i)) \cong \pi_1(A_S) * \pi_1(\hat D_i),
	\]
	where $\pi_1(D_i)/\pi_1(\partial D_i)$ denotes the quotient of $\pi_1(D_i)$ by the normal subgroup generated by $\pi_1(\partial D_i)$. Hence \eqref{eqn_AS_pi1_Di} is proved. 
%	To prove Equation \eqref{eqn_AS_pi1_Di}, note that since $A$ is simply connected, one can attach a collection of cells with dimensions at least $3$ to $A$ to obtain a contractible space $A^0$. 
%	Let $A^0_S = A^0 \cup (\cup_{i\in S} D_i)$. 
%	Since $A^0\setminus A$ consists of cells with dimensions at least $3$, we know that $\pi_1(A^0_S)\cong \pi_1(A_S)$. Since $A^0$ is contractible, $A^0_S$ is homotopy equivalent to $A^0_S/A^0$, which in turn is homeomorphic to $\vee_{i\in S} \hat D_i$. So 
%	\[
%	\pi_1(A_S)\cong \pi_1(A^0_S) \cong  \pi_1 (\vee_{i\in S} \hat D_i) \cong *_{i\in S}\, (\pi_1(\hat D_i)),
%	\]
%	and hence \eqref{eqn_AS_pi1_Di} is proved \footnote{Another way to prove \eqref{eqn_AS_pi1_Di} is to apply the van Kampen theorem repeatedly for attaching one $D_i$, $i\in S$, to $A$. Note that the embedding $\partial D_i\to A$ induces the zero map on the fundamental groups, while the map $j: \partial D_i \to D_i$ satisfies $\pi_1(D_i)/j_*(\pi_1(\partial D_i)) \cong \pi_1(\hat D_i)$.}. 
	
	Now we can prove the lemma. Let $S=\{1,\dots,k\}$, we have $S^4 = A_S$, and hence \eqref{eqn_AS_pi1_Di} shows that $\pi_1(\hat D_i)=1$ for all $i$. Applying \eqref{eqn_AS_pi1_Di} again, we conclude that $\pi_1(A_S)=1$ for all $S$. Since the set $A'$ in the statement of the lemma has the form $\Omega_S$ where $S$ contains all but one element of $\{1,\dots,k\}$, we know that $A'$ is simply connected.
\end{proof}

By Lemma \ref{lemm.reduction.to.connected} and the previous discussions, there exists a \emph{connected} stable minimal surface $M$ in $X$ such that $\iota(M)$ bounds a simply connected domain $A$ in $S^4$. Theorem \ref{theo.main.psc} then follows from the following result in topology. 
We will establish a stronger statement that not only determines the topology of $M$ but also finds the homeomorphism type of the domain $A$.

\begin{prop}
	\label{prop_topology_of_cut}
	Suppose $M$ is a connected smooth submanifold of $S^4$ such that one of the connected components of $S^4\setminus M$ is simply connected. Also assume that $M$ has the form
	\[
	M = \big( \#_i (S^3/\Gamma_i)\big)\# \big( \#^k S^1\times S^2\big)
	\]
	where $\Gamma_i$ acts freely and isometrically on $S^3$.
	Then $M \cong  \#^k S^1\times S^2$, and the closure of the simply connected component of $S^4\setminus M$ is homeomorphic to $\natural^k D^2\times S^2$. 
\end{prop}

The next lemma uses classical arguments to prove some basic properties of the homology groups of the complement of $M$. For its proof, we only need to assume that one of the components of $S^4\setminus M$ has trivial $H_1$. For later reference, we also allow $M$ to be a locally flat (but not necessarily smooth)  submanifold. 

\begin{lemm}
	\label{lem_Sigma_homology}
		Suppose $M$ is a locally flat connected oriented $3$-submanifold in $S^4$. Let $A$, $B$ denote the closures of the two components of  the complement of $M$. If $H_1(A)=0$, then
	\begin{enumerate}
		\item 		$H_1(M)$ has no torsion.
		\item  The map $H_2(M)\to H_2(A)$ induced by the inclusion is an isomorphism.
		\item $H_3(A) = 0$.
	\end{enumerate}
	
\end{lemm}

\begin{proof}
	(1) We follow an argument from \cite{hillman2024locally}. Consider the following commutative diagram
	\[\begin{tikzcd}
	H_2(A) \arrow{r} \arrow{d}& H_2(A,M) \arrow{d}\\
	H_2(S^4) \arrow{r} & H_2(S^4,B)
	\end{tikzcd}
	\]
	where the maps are induced by the inclusions of spaces. Since $H_2(S^4)=0$, and since by excision, the map $H_2(A,M) \to H_2(S^4,B)$ is an isomorphism, we conclude that the map 	$H_2(A) \to H_2(A,M) $ is zero. Now consider the homology long exact sequence
	\[
	H_2(A) \xrightarrow{0} H_2(A,M)\to H_1(M) \to H_1(A) \cong 0,
	\]
	we have $H_1(M)\cong H_2(A,M)$. By Lefschetz duality, we have $H_2(A,M) \cong H^2(A)$. By the universal coefficient theorem, the torsion of $H^2(A)$ is isomorphic to the torsion of $H_1(A)$, which is zero. So $H_1(M)$ has no torsion.

	(2) By the Mayer-Vietoris sequence
	\[
	0\cong H_3(S^4)\to H_2(M)\to H_2(A)\oplus H_2(B)\to H_2(S^4)\cong 0,
	\]
	we have $H_2(M)\cong H_2(A)\oplus H_2(B)$, so the map $H_2(M)\to H_2(A)$ induced by the inclusion is surjective. 
	On the other hand, by the exact sequence 
	\[
	H_3(A,M)\to H_2(M)\to H_2(A)
	\]
	and Lefschetz duality $H_3(A,M)\cong H^1(A) \cong 0$, we know that the map $H_2(M)\to H_2(A)$ is injective. Hence it is an isomorphism. 
	
	(3) By Lefschetz duality, $H_3(A)\cong H^1(A,M)$. By the cohomology long exact sequence for $(A,M)$, we have an exact sequence 
	\[
	H^0(A) \xrightarrow{\cong}  H^0(M) \to H^1(A,M) \to H^1(A)\cong 0,
	\]
	so $H^1(A,M) \cong 0$. 
\end{proof}

 We continue with the proof of Proposition \ref{prop_topology_of_cut}. 
 Let $A$, $B$ be the closures of the two components of the complement of $M$ such that $\pi_1(A)$ is trivial.
 By the assumptions of Proposition $\ref{prop_topology_of_cut}$, the manifold $M$ has the form $\big( \#_i (S^3/\Gamma_i)\big)\# \big( \#^k S^1\times S^2\big)$. 
 Write $M_1 =  \#_i (S^3/\Gamma_i)$, $M_2 = \#^k S^1\times S^2$. Following the usual convention, $M_1$ or $M_2$ is defined to be $S^3$ if there is no factor in the connected sum expression. 

In the connected sum decomposition of $M$, there are $k$ factors of $S^1\times S^2$. 
Let $\hat{A}$ be obtained from $A$ by attaching $k$ 3-handles to $\partial A = M$, such that the $k$ attaching spheres are given by $\{pt\}\times S^2$ in each component of $S^1\times S^2$. 

\begin{lemm}
	\label{lem_hat_A_contractible}
	$\partial \hat{A}$ is homeomorphic to $M_1$, and $\hat{A}$ is contractible. 
\end{lemm}

\begin{proof}
	$\partial \hat {A}$ is obtained from $\partial{A}$ by removing a tubular neighborhood of each attaching sphere and gluing 2 copies of $D^2$ for each removed neighborhood. So $\partial \hat{A}$ is homeomorphic to $M_1$.
	
	Now we show that $H_1(\hat{A}) \cong H_2(\hat{A}) \cong H_3(\hat{A}) \cong 0$. 
	
	Attaching 3-handles does not modify $H_1$, so we have $H_1(\hat {A}) \cong H_1(A)  \cong 0$.
	
	Consider the exact sequence
	\begin{equation}
	\label{eqn_hatA_exact_seq}
	H_3(A)\to H_3(\hat{A})\to H_3(\hat {A},A) \to H_2(A) \to H_2(\hat {A}) \to H_2(\hat {A},A).
	\end{equation}
	By excision, $H_2(\hat {A},A)\cong (H_2(D^3,\partial D^3))^k\cong 0$. Let $F$ be the closure of $\hat {A}\setminus A$ in $\hat{A}$. We have a commutative diagram
	\[\begin{tikzcd}
	H_3(\hat{A},A) \arrow{r} & H_2(A)\\
	H_3(F,\partial F) \arrow{r}\arrow{u}& H_2(\partial F) \arrow{u}
	\end{tikzcd}
	\]
	The vertical arrow from $H_3(F,\partial F)$ to $H_3(\hat{A},A)$ is an isomorphism because of excision. A direct computation shows that the horizontal arrow from $H_3(F,\partial F) $ to $H_2(\partial F)$ is an isomorphism. The vertical arrow from $H_2(\partial F)$ to $H_2(A)$ is the composition of 
	\[
	H_2(\partial F) \to H_2(M) \to H_2(A),
	\]
	where the first map is an isomorphism by Lemma \ref{lem_Sigma_homology}(1), the second map is an isomorphism by Lemma \ref{lem_Sigma_homology}(2). So the boundary map $H_3(\hat{A},A)\to H_2(A)$ is an isomorphism. 	By Lemma \ref{lem_Sigma_homology}(3), $H_3(A)\cong 0$.  So the exact sequence \eqref{eqn_hatA_exact_seq} becomes
	\[
	0\cong H_3(A)\to H_3(\hat{A})\to H_3(\hat {A},A) \xrightarrow{\cong} H_2(A) \to H_2(\hat {A}) \to H_2(\hat {A},A)\cong 0.
	\]
	Therefore, $H_3(\hat {A}) \cong H_2(\hat{A})\cong 0$. 
	
	Since $\hat{A}$ is a connected 4-manifold with non-empty boundary, we have $H_i(\hat {A}) = 0$ for all $i\ge 4$. 
	
	Since attaching 3-handles does not modify $\pi_1$, we know that $\hat A$ is simply connected. 
	
	Therefore, by the Hurewicz theorem, we know that $\pi_i(\hat A)$ is trivial for all $i$, so $\hat A$ is contractible. 
\end{proof}

The next lemma proves the first part of Proposition \ref{prop_topology_of_cut}. 
\begin{lemm}
	\label{lem_Sigma_1_S3}
	$M_1\cong S^3$.
\end{lemm}

\begin{proof}
	By Part (1) of Lemma \ref{lem_Sigma_homology}, we know that $H_1(M_1)$ has no torsion. 
	Therefore, $M_1 = aP\# b(-P)$, where $a,b$ are non-negative integers, $P=\Sigma(2,3,5)$ is the Poincar\'e homology sphere, and $(-P)$ is the manifold $P$ with the reversed orientation. 
	Now we follow an argument from \cite{chodosh2024complete}.
	By Lemma \ref{lem_hat_A_contractible}, $M_1$ represents the trivial element of the homology cobordism group $\Theta_{\mathbb{Z}}^3$. Since there exist homomorphisms from $\Theta_{\mathbb{Z}}^3$ to $\mathbb{Z}$ that takes the Poincar\'e homology sphere to non-zero integers (some examples of such homomorphisms are the Froyshov invariant, the d-invariant, and Manolescu's $\beta$-invariant), we have $a=b$. By \cite[Theorem 1.7]{taubes1987gauge}, we know that $aP\#a(-P)$ does not bound contractible smooth 4-manifolds unless $a=0$. Hence we have $M_1\cong S^3$.
\end{proof}

By Lemma \ref{lem_Sigma_1_S3}, we have $M \cong  \#^k S^1\times S^2$.
Now we prove the second part of Proposition \ref{prop_topology_of_cut}.

\begin{lemm}
	Let $M, A,B$ be as above. Then $A$ is homeomorphic to $\natural^k D^2\times S^2$. 
\end{lemm}

\begin{proof}
	By the previous results, $\partial A = M \cong \#^k S^1\times S^2$.
	Let $H_k = \natural^k D^3\times S^1$.
	Let $Z$ be the smooth manifold obtained by gluing $H_k$ with $A$ along a diffeomorphism of the boundaries. Since $\pi_1(\partial H_k) \to \pi_1(H_k)$ is surjective and $\pi_1(A)$ is trivial, the van Kampen theorem implies that $\pi_1(Z)=1$. 
	
	We claim that $H_2(Z)$ is trivial. Suppose $H_2(Z)\neq 0$, then there exists two closed oriented smooth 2-dimensional submanifolds $M_1$, $M_2$ of $Z$ such that $M_1$ and $M_2$ have a non-zero algebraic intersection number. Note that $H_k$ is a (closed) regular neighborhood of $\vee_k S^1$ in $Z$, so one can isotope $M_1$ and $M_2$ so that they are disjoint from $H_k$. Therefore, $M_1$ and $M_2$ are included in $A$ and have a non-zero algebraic intersection number. Since $A$ is embedded in $S^4$, this is impossible.
	
	By Freedman's classification theorem \cite[Theorem 1.5]{freedman1982topology}, $Z$ is homeomorphic to $S^4$. 
	
	There is a standard embedding of $H_k$ in $\mathbb{R}^4$ such that the closure of its complement is a punctured $\natural^k D^2\times S^2$.
	Since $Z$ is a smooth $4$--manifold and $H_k$ is a regular neighborhood of an embedded $1$--dimensional complex, and since $Z$ is simply connected, every smooth embedding of $H_k$ in $Z$ is smoothly isotopic to the standard embedding in a coordinate chart. Therefore, the complement of $\inte(H_k)$ in $Z$ is diffeomorphic to  $(\natural^k D^2\times S^2)\# Z$, which is homeomorphic to $\natural^k D^2\times S^2$. 
\end{proof}

This finishes the proof of Proposition \ref{prop_topology_of_cut}.
Now we carry out the proof of Theorem \ref{theo.main.biricci}.

\begin{proof}[Proof of Theorem \ref{theo.main.biricci}]
	Let $\hat M^3\subset (X^4,g)$ be a two-sided stable minimal hypersurface as given by Proposition \ref{prop.topology.extremal.minimal} and assume $\BiRic_g>0$. By Lemma \ref{lemm.reduction.to.connected} and the discussions above it, each connected component of $\iota(\hat M)$ separates $S^4$ into two connected components, and at least one of them is simply connected. Let $M$ be a connected component of $\hat M$. Denote by $\lambda_{\Ric}$ the smallest eigenvalue of the Ricci tensor of $M$. Then stability of $M$ implies that there exists a smooth function $W\ge -\lambda_{\Ric}$ such that
	\begin{equation}\label{eq.lambda.ricci.bound}
		\lambda_1(-\Delta - W) \ge \lambda>0
	\end{equation}
	on $M$, here $\lambda$ is a positive lower bound of the BiRicci curvature of $g$. Since the universal cover of $(M,g|_M)$ also satisfies \eqref{eq.lambda.ricci.bound}, \cite[Section 2]{ShenYe} implies that this universal cover has diameter bounded by $c/\sqrt{\lambda}$ for some constant $c>0$. In particular, $\pi_1(M)$ is finite and hence $M$ is diffeomorphic to $S^3/\Gamma$ for some discrete subgroup $\Gamma$ of $SO(4)$. Thus, Lemma \ref{lem_Sigma_homology} implies that $H_1(M)$ is torsion free, and hence $M$ is diffeomorphic to either $S^3$ or the Poincar\'e homology sphere. However, the latter cannot happen by Proposition \ref{prop_topology_of_cut}. 
\end{proof}

\section{A black hole topology theorem}\label{section.black.hole}
In this section we apply the proof of Theorem \ref{theo.main.psc} and Theorem \ref{theo.main.biricci} to obtain topological control of apparent horizons (quasilocal black hole boundaries) of certain time-symmetric $4$-dimensional initial data to the Einstein equations satisfying the dominant energy condition. To see its relations to our proof of Theorem \ref{theo.main.psc}, recall that such an initial data $(X^{n+1},g)$ necessarily satisfies that $R_g\ge 0$, and if $M^n=\partial X^{n+1}$ is an apparent horizon, then $M^n\subset X^{n+1}$ and $X^{n+1}$ contains no interior minimal hypersurfaces. By \cite{CaiGalloway2001,Galloway2008,GallowaySchoen2006}, $M^n$ is a stable minimal hypersurface and is necessarily Yamabe positive (that is, it is conformal to a PSC manifold). We are ready to prove Theorem \ref{theo.gr}.

\begin{proof}[Proof of Theorem \ref{theo.gr}]
	Let $E$ be an asymptotically flat end of $(X^4,g)$ and suppose $M$ is the apparent horizon in $E$. Then $M$ is Yamabe positive, and thus we have that
	\[M = \left(\#_i (S^3/\Gamma_i) \right)\# \left(\#^k S^2\times S^1\right). \]
	Since $E$ is asymptotically flat, it admits a mean convex foliation of three spheres outside a compact subset. Let $Y$ be a leaf in this foliation, and denote by $X'$ the region bounded between $M$ and $Y$. Note that $X'$ satisfies the assumptions of Theorem \ref{theo.main.psc}: it is a subset of $X$, which is assumed to have an embedding $\iota$ into $S^4$, and moreover  its complement in $S^4$ contains a homeomorphic $D^4$ (bounded by $\iota(Y)$ by the locally flat Schoenflies theorem).
	Since by assumption, $X'$ contains no interior stable minimal hypersurface that is homologous to $Y$, the proof of Proposition \ref{prop.topology.extremal.minimal} implies that $i_*:\pi_1(Y)\to \pi_1(X')$ is surjective.

	Consequently, we may apply the results in Section \ref{section.space.forms} to conclude that each connected component of $M$ is diffeomorphic to $S^3$ or a connected sum of $S^2\times S^1$'s.
\end{proof}

We remark that, other than assuming that $X$ is diffeomorphic to the interior of a manifold satisfying \ref{topo.assumption.on.domain}, the proof above is quite flexible. For instance, it suffices to assume that $E$ admits a mean convex foliation by mean convex three spheres outside a compact subset. In particular, any metric on $E$ that is $C^1$ asypmtotic to the Euclidean metric admits such a foliation.

Moreover, we may extend this result to complete noncompact manifolds $X^4$ with an end $E$ admitting a mean convex foliation of its infinity by $Y$, such that $\iota: X\to S^4$ satisfies that $S^4\setminus \iota(X)$ has a simply connected component bounded by $Y$. For instance, for all $k\in \ZZ_{>0}$, $Y = \#^k S^2\times S^1$ bounds the simply connected region $\natural^k S^2\times D^2$. Such manifolds appear as certain ALF or ALG gravitational instantons. An analogous argument proves that the outermost minimal hypersurface on this end is diffeomorphic to a disjoint union of $S^3$ or connected sums of $S^2\times S^1$. We will see an explicit example of such a manifold in Section \ref{subsection.DL.examples}.

\section{Extensions, examples and further questions}\label{section.example}

This section is devoted to several related discussions to our main results. We describe an extension of our results to construct locally minimizing hypersurfaces with controlled topology in a homeomorphic $B^4$ with an $S^2$ boundary. We also present a few examples and counterexamples on our main topological assumption \ref{topo.assumption.on.domain}: we construct a PSC embedding of the connected sums of lens spaces into a PSC $S^4$, illustrating the potential topological complexities of stable minimal hypersurfaces; we also discuss in detail the Dahl-Larsson construction \cite{DahlLarsson2019} of $S^2\times S^1$ horizons.

\subsection{Topology of the solution to Plateau problems in $D^4$}\label{subsection.plateau.minimal.hypersurface}

Consider a smooth homeomorphic closed 4-ball $D^4$ equipped with a Riemannian metric $g$. Let $\Sigma^2\subset S^3=\partial D^4$ be an embedded $S^2$. Assuming that $\partial D^4$ is weakly mean convex, we may use it as a barrier and solve the Plateau problem with boundary $\Sigma$. It is a very interesting question which $3$-manifolds with boundary may appear as such a Plateau solution - we believe that this is unknown even for the standard round ball in $\RR^4$ (but with an arbitrary $\Sigma = S^2$ in its boundary). We apply our results in Section \ref{section.minimal.hypersurface} to obtain some preliminary topology control of at least one such minimal hypersurface.

\begin{theo}\label{theo.topo.plateau.minimal}
	Let $X$ be a compact smooth manifold that is homeomorphic to $D^4$, $\Sigma^2\subset\partial X$ is an $S^2$. Suppose $g$ is a Riemannian metric on $X$ such that $\partial X$ is $g$-weakly mean convex. Then $\Sigma$ bounds a stable minimal hypersurface $M$ satisfying $\Tor(H_1(M)) = 0$.
\end{theo}

\begin{proof}
	We proceed as in \cite{White1992topology}. It is known that $\partial D^4\setminus \Sigma$ has two components $Y_1, Y_2$ and each $Y_i$, $i=1,2$, is diffeomorphic to $D^3$. Consider the space of stable minimal hypersurfaces $\cS$ in $(X,g)$ with boundary $\Sigma$ and $\cH^3$-volume bounded by $\cH^3(\partial X)$. Then $\cS$ is compact by standard curvature estimates. For an element $M\in \cS$, let $\Omega_M\subset X$ denote the region bounded by $M$ and $Y_1$. Then 
	\[\min\{\cH^4(\Omega_M): M\in \cS \} \]
	is achieved by $N\in \cS$. 
	
	By a similar proof as in Proposition \ref{prop.topology.extremal.minimal} (in this case, we may simply apply White's proof \cite{White1992topology} since $Y_1$ is simply connected and each $M\in \cS$
	 is connected, see Remark \ref{remark.simple.case.white}), we conclude that $\Omega_N$ is simply connected. Thus, applying Lemma \ref{lem_Sigma_homology} to $\partial \Omega_N = M\cup Y_1$, we conclude that $H_1(M)\cong H_1( M\cup Y_1)$ is torsion free.
\end{proof}

Inspired by Theorem \ref{theo.main.psc} and Theorem \ref{theo.main.biricci}, it would be very interesting to further narrow down the possibilities of the topology of $M$, perhaps with some additional curvature assumptions on $g$.

\begin{quest}
	Under some additional curvature assumptions on $g$, prove that $M$ has more restrictions on its topology. For example, is it true that if $g$ has nonnegative Ricci curvature and $\partial X$ is $g$-convex, we can find a Plateau minimal hypersurface $M$ that is diffeomorphic to $D^3$ or $\left(\#^k (S^2\times S^1)\right)\setminus D^3$?
\end{quest}

\subsection{A PSC embedding of connected sums of lens spaces}\label{subsection.connected.sum.lens}

Given a PSC metric $g$ on $S^4$, we are interested in which $3$-manifolds may appear as a stable minimal hypersurface in $(S^4,g)$. If $M^3\subset (S^4,g)$ is a stable minimal hypersurface, then necessarily $M = (S^3/\Gamma_i) \# \left(\#^k (S^2\times S^1)\right)$. On the other hand, it is known that not all $3$-manifolds in this form admits a smooth embedding into $S^4$. One necessary condition that dates back to Hantzsche \cite{Hantzsche1938} from 1938 states that for such an embedding to exist, $\Tor(H_1(M)) = G\oplus G$ for some abelian group $G$ (we again refer the readers to the survey article by Hillman \cite{hillman2024locally} for the collection of classical results of the embedding problem). The most basic example of such a $3$-manifold with nontrivial $\Tor(H_1(M))$ is the connected sum of two lens spaces, $L(p,q)\#-L(p,q)$ when $p$ is odd, as constructed by Zeeman \cite{zeeman1965twisting}. Below we show that such manifold does admit a PSC embedding into a PSC $S^4$.

Let $S^3 = \{(z,w)\in \mathbb{C}^2 \mid |z|^2 + |w|^2 =1\}$. Let $p\neq 0 $ and $p,q$ be coprime.  Consider the action of $\mathbb{Z}/p\mathbb{Z}$ on $S^3$ generated by 
\[
f: (z,w)\mapsto (e^{\frac{1}{p}\cdot {2\pi i}}z, e^{\frac{q}{p}\cdot 2\pi i}w).
\]
Then the action is free, and the quotient manifold is the lens space $L(p,q)$. Let $\sim$ denote the equivalence relation on $S^3$ induced by the action by $f$. Denote by $\tau: S^3\to S^3$ the conjugation map defined by $\tau(z,w) = (\bar z,\bar w)$.

The next lemma reinterprets a construction of Schubert \cite{schubert1956knoten} so that we can keep track of the curvature in Zeeman's construction. 
\begin{lemm}
	\label{lem_conj_on_L(p,q)}
	$\tau$ induces an involution on $L(p,q)$. The quotient space with respect to this involution is $S^3$, and the quotient map is a double branched cover. When $p$ is odd,  the branching locus is a smooth knot. 
\end{lemm}

\begin{proof}
	It is straightforward to check that $f\circ \tau = \tau \circ f^{-1}$, so $\tau$ induces an involution on $(S^3/\sim) = L(p,q)$. 
	
	Let $T_1 = \{(z,w)\in S^3\mid |z|\ge |w|\}$,  $T_2 = \{(z,w)\in S^3\mid |w|\ge |z|\}$. Then $f(T_1) = T_1$, $f(T_2) = T_2$. A fundamental domain of the action of $f$ on $T_1$ is 
	\[
	F_1 := \{(z,w)\in S^3 \mid arg(z) \in [0,2\pi/p], |z|\ge |w|\} \cong [0,2\pi/p] \times D^2,
	\]
	where $D^2$ is identified with the unit disk in $\mathbb{C}$, and the above diffeomorphism is defined by 
	\[
	(z,w) \mapsto (arg(z), \sqrt{2} w).
	\]
	The gluing map on $\partial F_1 \cong \partial [0,2\pi/p] \times D^2$ is given by 
	\[
	(0,x) \sim (2\pi/p, e^{\frac{q}{p}\cdot 2\pi i}x)
	\]
	for all $x\in D^2$, so $T_1/\sim$ is a solid torus $S^1\times D^2$. 
	
	We write down a diffeomorphism between $T_1/\sim$ and $S^1\times D^2$ explicitly. 
	Identify $S^1$ with the unit circle in $\mathbb{C}$. Let $\bar{\varphi}: F_1 \to S^1\times D^2$ be defined by 
	\[
	\bar \varphi(\theta, x) = (e^{i p \theta}, e^{-i q \theta} x),
	\]
	then $\bar\varphi$ induces a diffeomorphism $\varphi: (T_1/
	\sim) \to S^1\times D^2$. 
	
	The map $\tau$ induces an involution on $T_1/\sim$. The map is described on the fundamental domain as 
	\begin{align*}
	\tau |_{T_1/\sim} : [0,2\pi/p]\times D^2 / \sim & \to [0,2\pi/p]\times D^2 /\sim \\
	(\theta, x) & \mapsto (2\pi/p - \theta, e^{\frac{q}{p}\cdot 2\pi i}\bar{x}).  
	\end{align*}
	The fixed point set of $\tau |_{T_1/\sim}$ consists of two arcs, represented in the fundamental domain by
	\begin{equation}
	\label{eqn_branching_arcs_T1}
	\{\pi/p\}\times \{re^{\frac{q}{p}\cdot \pi i}\mid r\in [-1,1]\} \quad\text{and}\quad \{0\}\times \{r\mid r\in [-1,1]\}.
	\end{equation}
	
	Note that 
	\begin{align*}
	\varphi\circ (\tau|_{T_1/\sim}) \circ \varphi^{-1} (e^{ip\theta}, e^{-i q \theta} x) 
	& = (e^{ip(2\pi/p - \theta)}, e^{-iq((2\pi/p - \theta)}  e^{\frac{q}{p}\cdot 2\pi i}\bar{x}) \\
	& = (e^{-ip\theta}, e^{i q \theta} \bar{x}).
	\end{align*}
	So the action of $\varphi\circ (\tau|_{T_1/\sim}) \circ \varphi^{-1}$ on $S^1\times D^2$ is the product of two  reflections on $S^1$ and $D^2$. Therefore, the quotient space is homeomorphic to $D^3$ and the quotient map is a double branched cover.
	
	Similarly, the involution $\tau$ defines a double branched cover from $T_2/\sim$ to $D^3$. As a result, the quotient map $L(p,q)\to L(p,q)/\tau$ is a double branched cover and we have $L(p,q)/\tau\cong S^3$.
	
	The branching locus is the image of four arcs on $L(p,q)$, which are given by \eqref{eqn_branching_arcs_T1} and a similar formula on $T_2/\sim$.  To write down the formula for the branching locus on $T_2/\sim$, let 
	\[
	F_2 := \{(z,w)\in S^3 \mid arg(w) \in [0,2\pi/p], |w|\ge |z|\} \cong D^2\times [0,2\pi/p],
	\]
	be a fundamental domain of $T_2/\sim$, where the diffeomorphism is given by $(z,w)\mapsto (\sqrt{2} z, \arg(w))$. Let $q'$ be an integer such that $p\mid qq'-1$. Then the branching locus on $T_2/\sim$ are given by the following two arcs on $F_2$:
	\begin{equation}
	\label{eqn_branching_arcs_T2}
	\{re^{\frac{q'}{p}\cdot \pi i}\mid r\in [-1,1]\} \times \{\pi/p\}  \quad\text{and}\quad  \{r\mid r\in [-1,1]\}\times \{0\}.
	\end{equation}
	The endpoints of the four arcs have the following preimages in $S^3$:
	\begin{enumerate}
		\item The endpoints of $\{\pi/p\}\times \{re^{\frac{q}{p}\cdot \pi i}\mid r\in [-1,1]\}$ in \eqref{eqn_branching_arcs_T1} lift to $\frac{\sqrt{2}}{2} ( e^{\pi i/p}, \pm e^{\frac{q}{p}\cdot \pi i})$.
		\item The endpoints of $\{0\}\times \{r\mid r\in [-1,1]\}$ in \eqref{eqn_branching_arcs_T1} lift to $\frac{\sqrt{2}}{2} ( 1, \pm 1)$.
		\item The endpoints of $\{re^{\frac{q'}{p}\cdot \pi i}\mid r\in [-1,1]\} \times \{\pi/p\} $ in \eqref{eqn_branching_arcs_T2} lift to $\frac{\sqrt{2}}{2} ( \pm e^{\frac{q'}{p}\cdot \pi i}, e^{\pi i/p})$.
		\item The endpoints of $\{r\mid r\in [-1,1]\}\times \{0\}$ in \eqref{eqn_branching_arcs_T2} lift to $\frac{\sqrt{2}}{2} (\pm 1, 1)$.
	\end{enumerate}
	Since $p$ is odd, we may assume without loss of generality that both $q'$ and $q$ are odd. Then $2p\mid qq'-1$. In the following, we abuse notation and extend the action of $f$ to $\mathbb{C}^2$ by the same formula.  Then
	\[
	f^{(q'-1)/2} ( e^{\pi i/p},  e^{\frac{q}{p}\cdot \pi i}) = (  e^{\frac{q'}{p}\cdot \pi i}, e^{\frac{q}{p}\cdot \pi i}e^{\frac{q'-1}{2} \frac{q}{p}\cdot 2\pi i}) =  (e^{\frac{q'}{p}\cdot \pi i}, e^{\pi i/p}),
	\]
	\[
	f^{(p-1)/2} ( e^{\pi i/p},  -e^{\frac{q}{p}\cdot \pi i}) = (  e^{\frac{p}{p}\cdot \pi i},- e^{\frac{q}{p}\cdot \pi i}e^{\frac{p-1}{2} \frac{q}{p}\cdot 2\pi i}) =  (-1, 1),
	\]
	\[
	f^{(p+q')/2} (1,-1) = (  e^{\frac{p+q'}{p}\cdot \pi i},- e^{\frac{q(p+q')}{p}\cdot \pi i}) =  ( -e^{\frac{q'}{p}\cdot \pi i}, e^{\pi i/p}).
	\]
	As a result, the four arcs above connect to form a smooth knot in $L(p,q)$. 
\end{proof}

We are now ready to state our main construction.  We show that Zeeman's construction can be performed preserving the PSC conditions.

\begin{prop}\label{prop.psc.embedding.of.lens}
	Let $p, q$ be coprime integers and $p$ is odd. There exists an embedding of $M=L(p,q)\# - L(p,q)$ into  $S^4$ equipped with a Riemannian metric $g$, such that $g$ has positive scalar curvature, and the induced metric $g|_M$ has positive scalar curvature.
\end{prop}

\begin{proof}
	Let $N= L(p,q)$ be a lens space with $p$ odd, and let $\bar{\tau}:L(p,q)\to L(p,q)$ be the involution given by Lemma \ref{lem_conj_on_L(p,q)}. Equip $L(p,q)$ with the standard round metric $g_0$ inherited from the round $S^3$. Lemma \ref{lem_conj_on_L(p,q)} implies that $\bar\tau$ is an isometry of $g_0$. Let $x\in L(p,q)$ be a point on the branching locus of $\tau$, and for small $\rho>0$, $D_\rho(x)$ be the geodesic ball at $x$. We now fix $\rho>0$ sufficiently small, so we may apply the classical result of Gromov-Lawson \cite{GromovLawson80} (the point $\{x\}$ has codimension $3$ in $N$), and deform the metric $g_0$ in $D_\rho\setminus\{x\}$ to another PSC metric $g_1$, such that $g_1$ is a product metric in $D_{\rho/2}\setminus \{x\}$, while keeping $\tau$ an isometry of $g_1$. Let
	\[M_\tau = \left(N \setminus D_\rho(x)\right)\times [0,1] / (y,0)\sim (\tau (y),1)\]
	be the mapping torus of $\tau$ on $N\setminus D_\rho(x)$. The product metric $g_1+ dt^2$ on $(N\setminus D_\rho(x))\times [0,1]$ induces a metric $\hat g$ with positive scalar curvature and is product near  $\partial M_\tau=S^2\times S^1$. Thus, we may glue onto $M_\tau$ a  $S^2\times D^2$ with positive scalar curvature, equipped with a product metric near its boundary. The gluing is defined by a fiber preserving diffeomorphism $\varphi$ of $S^2\times S^1$. A classical result of Zeeman\footnote{In fact, Zeeman's theorem states that the same result holds for every cyclic branched cover of $S^3$ along a knot.} \cite{zeeman1965twisting} states that by suitably choosing $\varphi$, the result
	\[M_\tau \cup_{\varphi}( S^2\times D^2)\]
	is diffeomorphic to the standard $S^4$. Note that the metrics glue smoothly into $g$ with positive scalar curvature.
	
	Taking any $t\in [0,1]$ in the $[0,1]$--coordinate of $M_\tau$, we obtain a smooth PSC embedding of
	\[L(p,q)\setminus D_\rho(x) \hookrightarrow M_\tau \hookrightarrow (S^4,g)\]
	such that the induced metric is product near $\partial (L(p,q)\setminus D_\rho(x))$. Therefore, by taking the boundary of a small tubular neighborhood of $L(p,q)\setminus D_\rho(x)$ (see, e.g. \cite[Theorem 5.7]{GromovLawson80}) -- which is diffeomorphic to $M=L(p,q)\#-L(p,q)$ --  we obtain a smooth embedding of $M$ into $(S^4,g)$ such that the induced metric has positive scalar curvature.
\end{proof}

Proposition \ref{prop.psc.embedding.of.lens} motivates the following natural question.

\begin{quest}
	Construct a stable minimal embedding of $M=L(p,q)\# - L(p,q)$ into some PSC Riemannian $(S^4,g)$.
\end{quest}

The example constructed in Proposition \ref{prop.psc.embedding.of.lens} is minimal (in fact, the metric is a product nearby) everywhere except for the boundary of the tubular neighborhood of $\partial (L(p,q)\setminus D_\rho(x))$, where the mean curvature is large. With some trivial modifications of this construction, $M$ can be made everywhere mean convex.

\subsection{The Dahl-Larsson example}\label{subsection.DL.examples}

In \cite{DahlLarsson2019}, Dahl-Larsson constructed, for all $n\ge 3$, an asymptotically flat manifold $(X^n,g)$ with zero scalar curvature such that its apparent horizon is diffeomorphic to the unit normal bundle of a codimension $\ge 3$ submanifold $N$. We briefly discuss their discussion in the case we consider in our paper.

Let $\delta$ be the flat metric on $\RR^4$, $\gamma\subset \RR^4$ be an embedded curve.  Consider the Green's function 
\[G(x) = \int_\gamma |x-y|^{-2}dy\]
with poles along $\gamma$. It is well-known that $G$ satisfies the asymptotics
\[G(x)\sim \frac{A}{|x|_\delta^2}, ~~ |x|\to\infty; \quad G(x)\sim \frac{B}{\dist_\delta(x,\gamma)},~~x\to \gamma,\]
For some constants $A, B$. For $\eps>0$, consider the conformally deformed metric
\[g_\eps = (1+\eps G)^2 \delta. \]
The manifold $(\RR^4,g_\eps)$ is complete, noncompact and has two ends: $E_1$ near the infinity of $\RR^4$, and $E_2$ near $\gamma$. For each $\eps>0$, $g_\eps$ has zero scalar curvature since $\Delta_\delta u_\eps = 0$. $g_\eps$ on $E_1$ is asymptotically flat. $g_\delta$ on $E_2$ enjoys a similar asymptotic decay. To see this, let $s$ be the arclength coordinate of $\gamma$, $t$ the distance function to $\gamma$. Then locally we may write the metric $\delta = ds^2 + dt^2 + t^2 g_{s,t}$, where $g_{s,t}$ is a metric on the unit $S^2$ converging smoothly to the round metric as $s,t\to 0$. Setting $r=-\log t$, we see that
\begin{align*}
	g_\eps &= (1+\eps e^r)^2 (ds^2 + e^{-2r}dr^2 + e^{-2r} g_{s,r})\\
				& = (e^{-r}+\eps)^2 \left(e^{2r}ds^2 +dr^2 + g_{S^2}\right) + o(1).
\end{align*}
Therefore, $g_\eps$ is asymptotic to the product metric on $S^2\times \HH^2$ on $E_2$. In particular, near the infinity of $E_2$, it admits a mean convex foliation by $S^2\times S^1$.

\begin{theo}[{\cite[Theorem 1.1]{DahlLarsson2019}}]\label{theo.dahl.larsson}
	For sufficiently small $\eps>0$, it holds that the outermost apparent horizon $M_\eps$ of $(\RR^4,g_\eps)$ in $E_1$ is diffeomorphic to $S^2\times S^1$. In fact $M_\eps$ is the graph of a smooth function on the unit normal bundle in normal coordinates for $\gamma$.
\end{theo}

The proof of Theorem \ref{theo.dahl.larsson} is delicate. They constructed mean convex barriers and forced the apparent horizon in $E_1$ to lie between the tubular hypersurfaces around $\gamma$ of $\delta$-distance $C_{inner}\eps$ and $C_{outer}\eps$, provided that $\eps>0$ is sufficiently small.

We observe that the same proof in fact also proves the following simpler statement:

\begin{prop}
	For sufficiently small $\eps>0$, the outermost minimal hypersurface on $E_2$ is diffeomorphic to $S^2\times S^1$, and is the graph of a smooth function on the unit normal bundle in normal coordinates for $\gamma$.
\end{prop}

The Dahl-Larsson example naturally leads to the following question.

\begin{quest}
	For each integer $l>1$, construct an asymptotically flat manifold  $(\RR^4\setminus \gamma, g)$ with $R_g\ge 0$ such that the outermost minimal hypersurface is diffeomorphic to $\#^l (S^2\times S^1)$.
\end{quest}

%\subsection{Black lenses by Khuri-Rainone}\label{subsection.black.lenses}

\bibliographystyle{amsplain}
\bibliography{bib.bib}

\end{document}